\documentclass[11pt]{amsproc}

\usepackage{amssymb}
\usepackage{graphicx}
\usepackage[all]{xypic}

\usepackage{pdfsync}

\SelectTips{eu}{11}
\UseTips

\newtheorem{theorem}{Theorem}[section]

\newtheorem{lemma}[theorem]{Lemma}
\newtheorem{question}[theorem]{Question}

\newtheorem{definition}[theorem]{Definition}

\renewcommand{\setminus}{\smallsetminus}

\newcommand{\mono}{\rightarrowtail}
\newcommand{\epi}{\twoheadrightarrow}
\newcommand{\iso}{\cong}
\newcommand{\FF}{\mathfrak F}

\newcommand{\Q}{{\mathbb Q}}
\newcommand{\Z}{{\mathbb Z}}
\newcommand{\R}{{\mathbb R}}

\newcommand{\N}{{\mathbb N}}
\newcommand{\F}{{\mathbb F}}

\newcommand{\clh}{\operatorname{\scriptstyle\bf H}}
\newcommand{\cll}{\operatorname{\scriptstyle\bf L}}
\newcommand{\ho}{\clh_1}
\newcommand{\hf}{\clh\FF}
\newcommand{\lhf}{\cll\clh\FF}
\newcommand{\hof}{\ho\FF} 
\newcommand{\fpinfty}{\operatorname{FP}_\infty}
\newcommand{\fp}{\operatorname{FP}}

\newcommand{\Ext}{\operatorname{Ext}}
\newcommand{\Hom}{\operatorname{Hom}}

\newcommand{\cont}{\operatorname{\mathcal F_0}}
\newcommand{\colim}{{\displaystyle\lim_{\buildrel\longrightarrow\over\lambda}\ }}
\newcommand{\colimn}{{\displaystyle\lim_{\buildrel\longrightarrow\over n}\ }}
\newcommand{\colims}%
{{\displaystyle\lim_{\buildrel\longrightarrow\over{s\in S}}\ }}
\newcommand{\colimf}%
{{\displaystyle\lim_{\buildrel\longrightarrow\over{H\in\mathcal F}}\ }}
\newcommand{\projdim}{\operatorname{proj.\,dim}}

\newcommand{\silp}{\operatorname{silp}}

\newcommand{\Mo}{\operatorname{Mod}}
\newcommand{\HYPHEN}{\operatorname{-}}
\newcommand{\Mod}{\operatorname{\Mo\HYPHEN}}

\newcommand{\blah}{{\phantom M}}

\title[]%
{Groups with many finitary cohomology functors}
\author[P. H. Kropholler]{P. H. Kropholler}
\address{School of Mathematics and Statistics, University of Glasgow, Glasgow G12 8QQ}
\email{peter.kropholler@glasgow.ac.uk}

\subjclass[2000]{20J06, 20J05, 18G15}
\keywords{cohomology of groups, finitary functors.}

\begin{document}

\begin{abstract}
For a group $G$, we study the question of which cohomology functors commute with all small filtered colimit systems of coefficient modules. We say that the functor $H^n(G,\blah)$ is {\em finitary} when this is so and we consider the {\em finitary set} for $G$, that is, the set of natural numbers for which this holds. It is shown that for the class of groups $\lhf$ there is a dichotomy: the finitary set of such a group is either finite or cofinite. We investigate which sets of natural numbers $n$ can arise as finitary sets for suitably chosen $G$ and what restrictions are imposed by the presence of certain kinds of normal or near-normal subgroups. Although the class $\lhf$ is large, containing soluble and linear groups, being closed under extensions, subgroups, amalgamated free products, HNN-extension, there are known to be many not in $\lhf$ such as Richard Thompson's group $F$. Our theory does not extend beyond the class $\lhf$ at present and so it is an open problem whether the main conclusions of this paper hold for arbitrary groups. There is a survey of recent developments and open questions.
\end{abstract}

\maketitle

\section*{Organizational Statement}

This paper lays the foundation stones for a series of papers by the author's former student Martin Hamilton: \cite{hamilton2011,hamilton2008,hamilton2009}. As sometimes happens, this literature has not been published in the order in which it was intended to be read and for this reason I am taking the opportunity of this conference proceedings to include a survey of Hamilton's papers and a discussion of possible future directions. This survey follows and expands upon the spirit of the talk I gave at the meeting. The present work also lays the foundations for Hamilton's results in \cite{hamilton2011}. The papers \cite{hamilton2008} and \cite{hamilton2009} build on the results of the present work and \cite{hamilton2011}. 

\section{Introduction}

Groups of types $\fp$, $\fpinfty$ or $\fp_n$ have been widely explored. The properties are most often described in terms of projective resolutions. A group $G$ has type $\fp_n$ if and only if there is a projective resolution $\cdots\to P_j\to P_{j-1}\to\dots\to P_1\to P_0\epi\Z$ of the trivial module $\Z$ over the integral group ring $\Z G$ such that $P_j$ is finitely generated for $j\le n$. Type $\fp_1$ is equivalent to finite generation of the group, and for finitely presented groups, type $\fp_n$ is equivalent to the existence of an Eilenberg--Mac Lane space with finite $n$-skeleton. These properties can also be formulated in terms of cohomology functors by using the notion of a {\em finitary functor}.
A functor is said to be {\em finitary} if it preserves filtered colimits (see \S6.5 of \cite{leinster}; also \S3.18 of \cite{adamekrosicky}). For a group $G$ and a natural number $n$ we can consider whether or not the $n$th cohomology functor is finitary. For our purposes it is also useful to consider additive functors $F$ between abelian categories with the property that
$$\lim_\to F(M_\lambda)=0$$
whenever $(M_\lambda)$ is a filtered colimit system satisfying $\displaystyle\lim_\to M_\lambda=0$: we shall say that $F$ is {\em $0$-finitary} when this condition holds.
Here is a classical result of Brown \cite{Brown1975} phrased in this language. The details and similar results of Bieri and Eckmann can be found in (\cite{bieri-qmw} Theorem 1.3) and (\cite{Brown-book} VIII Theorem 4.8).

It would seem very natural to use the terminology \emph{continuous} to mean finitary and \emph{continuous at zero} to mean $0$-finitary. We have not done so because the terminology now standard in category theory reserves the use of the term continuous for functors which commute with limits while finitary refers to functors which commute with colimits. 

\begin{lemma}\label{classical}
For a group $G$ and $n\ge0$, the following are equivalent:
\begin{enumerate}
\item
$G$ is of type $\fp_n$;
\item
$H^i(G,\blah)$ is finitary for all $i<n$;
\item
$H^i(G,\blah)$ is $0$-finitary for all $i\le n$;
\end{enumerate}
\end{lemma}

\begin{definition}
We write $\mathcal F(G)$ (resp. $\cont(G)$) for the set of positive natural numbers $n$ for which 
$H^n(G,\blah)$ is finitary (resp. $0$-finitary).
\end{definition}

\section{Main Theorems}

Our basic result concerns groups in the class $\lhf$ as described in \cite{kropholler1993} and \cite{krmi}. We write $\N^+$ for the set of natural numbers $n\ge1$.

\begin{theorem}\label{basic1}
Let $G$ be an $\lhf$-group for which $\cont(G)$ is infinite. Then
\begin{enumerate}
\item $\cont(G)$ is cofinite in $\N^+$;
\item there is a bound on the orders of the finite subgroups of $G$;
\item there is a finite dimensional model for the classifying space $\underbar EG$ for proper group actions.
\end{enumerate}
\end{theorem}

We refer the reader to \cite{krmi} for a brief explanation of the classifying space $\underbar EG$, and to L\"uck's survey article \cite{Lueck-survey} for a comprehensive account.
Our theorem shows that for any $\lhf$-group $G$ the set  
$\cont(G)$ is either finite or cofinite in the set $\N^+$ of positive natural numbers. 
It is unknown whether there exists a group $G$ outwith the class $\lhf$ for which $\cont(G)$ is a moiety (i.e. neither finite nor cofinite). Notice that groups of finite cohomological dimension all belong to $\lhf$ and have a cofinite invariant because almost all their cohomology functors vanish. On the other hand the theorem shows that $\mathcal F_0(G)$ is finite for all torsion-free $\lhf$ groups of infinite cohomological dimension. Both conditions (ii) and (iii) above are highly restrictive. However, the theorem does not give a characterization for cofiniteness of $\cont(G)$ for groups with torsion: this turns out to be a delicate question even for abelian-by-finite groups and is studied by Hamilton in the companion article \cite{hamilton2011}. Before turning to the proof of Theorem \ref{basic1} we show that $\cont(G)$ can behave in any way subject to the constraints it entails.

\begin{theorem}\label{basic2}
Given any finite or cofinite subset $S\subseteq\N^+$ there exists a group $G$ such that
\begin{enumerate}
\item $\cont(G) = S$;
\item $G$ has a finite dimensional model for the classifying space $\underline EG$.
\end{enumerate}
\end{theorem}
Note that all groups with finite dimensional models for $\underline E$ belong to $\hof\subset\lhf$. 
There is an abundance of examples of groups satisfying various homological finiteness conditions and we can select examples easily to establish Theorem \ref{basic2}. 

\begin{lemma}\label{basic3}\ 
\begin{enumerate}
\item For each $n$ there is a group $J_n$ such that $\cont(J_n)=\N^+\setminus\{n\}$.
\item For each $n$ there is a group $H_n$ of type $\fp_n$ such that
$\cont(H_n)$ is finite.
\end{enumerate}
Moreover, the groups  here can be chosen to have finite dimensional models for their classifying spaces for proper actions.
\end{lemma}
\begin{proof}
For $J_n$ we can take Bieri's example $A_{n-1}$  of a group which is of type $\fp_{n-1}$ but not of type $\fp_n$ (\cite{bieri-qmw} Proposition 2.14). This group has cohomological dimension $n$ and hence it has the desired properties.
Many more examples like this can be obtained using the powerful results of Bestvina and Brady \cite{bestvinabrady1997}

For the groups $H_n$ we may choose Houghton's examples \cite{houghton} of groups which were shown to be of type $\fp_{n}$ but not type $\fp_{n+1}$ by Brown \cite{Brown1987}.  The group $H_n$ is defined to be the group of those permutations $\sigma$ of
$\{0,1,2,\dots,n\}\times\N$
for which there exists $m_0,\dots,m_n\in\N$ (depending on $\sigma$) such that
$\sigma(i,m)=(i,m+m_i)$
for all but finitely many ordered pairs $(i,m)$. The vector $(m_0,\dots,m_n)$ is uniquely determined by $\sigma$ and necessarily satisfies $m_0+\dots+m_n=0$. Thus there is a group homomorphism
$H_n\to\Z^{n+1}$
given by $\sigma\mapsto(m_0,\dots,m_n),$
and $H_n$ is a group extension
$$T\mono H_n\epi\Z^n$$
where $T$ is the group of all finitary permutations of the countably infinite set $\{0,1,2,\dots,n\}\times\N$. 
We will describe an explicit construction for a finite dimensional $\underline EH_n$. Let $T_0<T_1<T_2<\dots<T_i<\cdots$ be a chain of finite subgroups of the locally finite group $T$, indexed by $i\in\N$ and having $\bigcup  T_i=T$. Let $\Gamma$ be the graph whose edge and vertex sets are the cosets of the $T_i$:
$$V:=\bigsqcup T_i\backslash T=:E$$ and in which the terminal and initial vertices of an edge $e=T_ig$ are $\tau e=T_{i+1}g$ and $\iota e=T_ig$. Then $\Gamma$ is a $T$-tree and its realization as a one dimensional $CW$-complex is a one dimensional model $X$ for $\underline ET$. Now take any $H_n$-simplicial complex abstractly homeomorphic to $\R^n$ on which $T$ acts trivially and on which the induced action of $H_n/T$ is free. Then we can thicken the space $X$ by replacing each vertex by a copy of $T$ appropriately twisted by the action of $H_n$ and replacing each higher dimensional simplex of $X$ by the join of the trees placed at its vertices. This creates a finite dimension model for $\underbar EH_n$. As well as establishing (iii), this also shows that $H_n$ belongs to $\hof$ and hence Theorem \ref{basic1} applies. Since $T$ is an infinite locally finite group we see that the conclusion Theorem 2.1(ii) fails and it follows that $\cont(H_n)$ is finite as required.
\end{proof}

\begin{lemma}\label{basic4}
Suppose that $G$ is the fundamental group of a finite graph of groups in which the edge groups are of type $\fpinfty$. Then $\cont(G)=\bigcap\cont(G_v)$, the intersection of the finitary sets of vertex stabilizers $G_v$ as $v$ runs through a set of orbit representatives of vertices.
\end{lemma}
\begin{proof}
The Mayer--Vietoris sequence for $G$
is a long exact sequence of the form
$$\dots\to\prod H^{n-1}(G_e,\ \ )\to H^n(G,\ \ )\to \prod H^n(G_v,\ \ )\to
\prod H^n(G_e,\ \ )\to\dots$$
Here, $e$ and $v$ run through sets of orbit representatives of edges and vertices, and since $G$ comes from a {\em finite} graph of groups, the product here are finite. Since the edge groups $G_e$ are $\fpinfty$, we find that restriction induces an isomorphism
$$\colim H^n(G,M_\lambda)\to\prod\colim H^n(G_v,M_\lambda)$$
whenever $(M_\lambda)$ is a vanishing filtered colimit system of $\Z G$-modules.
Thus if $n\notin\cont(G)$ then any system $(M_\lambda)$ witnessing this must also bear witness to a infinitary functor $H^n(G_v,\ \ )$ for some $v$, and we see that
$$\cont(G)\supseteq\bigcap\cont(G_v).$$
On the other hand, if $n\notin\bigcap\cont(G_v)$ then there is a $v$ and a vanishing filtered colimit system $(U_\lambda)$ of $\Z G_v$-modules such that
$$\colim H^n(G_v,U_\lambda)\ne0.$$ Set $M_\lambda:=U_\lambda\otimes_{\Z G_v}\Z G$. Since, {\em qua $\Z G_v$-module,} $U_\lambda$ is a natural direct summand of $M_\lambda$ we also have
$$\colim H^n(G_v,M_\lambda)\ne0$$ and therefore from the isomorphism
$$\colim H^n(G,M_\lambda)\ne0$$ and $n\notin\cont(G)$. Thus 
$\cont(G)\subseteq\bigcap\cont(G_v)$ and the result is proved.
\end{proof}

The simplest way to apply this is to a free product of finitely many groups. We deduce that the collection of subsets which can arise as $\cont(G)$ for some $G$ is closed under finite intersections. 

\begin{proof}[Proof of Theorem \ref{basic2}]
Suppose that $S$ is a cofinite subset of $\N^+$. Then we take $G$ to be the free product of the finitely many groups $J_n$, as described in Lemma \ref{basic3}, for which $n\notin S$. Lemmas \ref{basic3} and \ref{basic4} show that $\cont(G)=S$.

On the other hand, if $S$ is finite, then choose an $n\in \N^+$ greater than any element of $S$ and let $G$ be the free product of the group $H_n$ and the groups $J_m$ for $m\in S$. Then again, Lemmas \ref{basic3} and \ref{basic4} show that $\cont(G)=S$. 

That the groups constructed this way have finite dimensional models for their classifying spaces follows from the easy result below.
\end{proof}

\begin{lemma}
Let $G$ be a finite free product $K_1*\dots*K_n$ where each $K_i$ has a finite dimensional $\underline EK_i$. Then $G$ also has a finite dimensional $\underline EG$.
\end{lemma}
\begin{proof}
Choose a $G$-tree $T$ whose vertex set $V$ is the disjoint union of the $G$-sets
$$K_i\backslash G:=\{K_ig:\ g\in G\}$$
and so that $G$ acts freely on the edge set $E$. 
\end{proof}

In order to prove Theorem \ref{basic1} we shall make use of complete cohomology: we shall use Mislin's definition in terms of satellite functors. Let $M$ be a $\Z G$-module. We write $FM$ for the free module on the underlying set of non-zero elements of $M$. The inclusion $$M\setminus\{0\}\to M$$ induces a natural surjection $$FM\to M$$ whose kernel is written $\Omega M$. Both $F$ and $\Omega$ are functorial: for a map $\theta:M\to N$, the induced map $F\theta:FM\to FN$ carries elements $m\in M\setminus\ker\theta$ to their images $\theta m\in N$ and carries elements of $\ker\theta\setminus\{0\}$ to $0$. The functor $F$ is left adjoint to the forgetful functor from $\Z G$-modules to pointed sets which forgets everything save the set and zero. The advantage of working with $F$ rather than simply using the free module on the underlying set of $M$ is that it is $0$-finitary. Our functor $\Omega$ inherits this property: it is also $0$-finitary. We shall make use of these observations in proving Theorem \ref{basic1}. As in \cite{mislin1994} the $j$th complete cohomology of $G$ is given by the colimit:
$$\widehat H^j(G,M):=\colimn H^{j+n}(G,\Omega^nM).$$
\begin{lemma}\label{mislin1994}
If there is an $m$ such that $H^j(G,F)=0$ for all free modules $F$ and all $j\ge m$ then the natural map
$$H^j(G,\blah)\to\widehat H^j(G,\blah)$$ is an isomorphism for all $j\ge m+1$.
\end{lemma}
\begin{proof}
The connecting maps
$H^{j+n}(G,\Omega^nM)\to H^{j+n+1}(G,\Omega^{n+1}M)$
in the colimit system defining complete cohomology are all isomorphisms because they fit into the cohomology exact sequence with
$H^{j+n}(G,F\Omega^nM)$ and
$H^{j+n+1}(G,F\Omega^{n}M)$
to the left and the right, and these both vanish for $j\ge m+1$.
\end{proof}

\begin{theorem}\label{main}
Let $G$ be an $\lhf$-group for which the complete cohomology functors $\widehat H^j(G,\ \ )$  are $0$-finitary for all $j$. Then
\begin{enumerate}
\item The set $B$ of bounded $\Z$-valued functions on $G$ has finite projective dimension.
\item If $M$ is a $\Z G$-module whose restriction to every finite subgroup is projective then $M$ has finite projective dimension: in fact
$$\projdim M\le\projdim B.$$
\item For all $n>\projdim B$, $H^n(G,\ \ )$ vanishes on free modules.
\item For all $n>\projdim B$, the natural map $H^n(G,\ \ )\to\widehat H^n(G,\ \ )$ is an isomorphism.
\item $n\in\cont(G)$ for all $n>\projdim B$.
\item $G$ has rational cohomological dimension $\le\projdim B+1$.
\item There is a bound on the orders of the finite subgroups of $G$.
\item There is a finite dimensional model for $\underline EG$. 
\end{enumerate}
\end{theorem}
\begin{proof}[Outline of the proof.] Since $\widehat H^j(G,\ \ )$ is finitary and $G$ belongs to the class $\lhf$ we have the following algebraic result about the cohomology of $G$:
\begin{center}
$\widehat H^j(G, B)=0$ for all $j$.
\end{center}
For $\hf$-groups of type $\fpinfty$ this follows from (\cite{cokr2}, Proposition 9.2) by taking the ring $R$ to be $\Z G$ and taking the module $M$ to be the trivial $\Z G$-module $\Z$. However we need to strengthen this result in two ways. Firstly we wish to replace the assumption that $G$ is of type $\fpinfty$ by the weaker condition that the functors $\widehat H^j(G,\ \ )$  are $0$-finitary for all $j$. This presents no difficulty because the proofs in \cite{cokr2} depend solely on calculations of complete cohomology rather than ordinary cohomology. The second problem is also easy to address but we need to take care. Groups of type $\fpinfty$ are finitely generated and so $\lhf$-groups of type $\fpinfty$ necessarily belong to $\hf$. However the weaker condition that the complete cohomology is finitary does not imply finite generation: for example, all groups of finite cohomological dimension have vanishing complete cohomology and there exists such groups of arbitrary cardinality. A priori we do not know that $G$ belongs to $\hf$ and we must reprove the result that
$$\widehat H^*(G,B)=0$$ from scratch. The key, which has been established \cite{brianmatthews} by Matthews, is as follows:

\begin{lemma}\label{new}
Let $G$ be an group for which all the functors $\widehat H^j(G,\blah)$ are $0$-finitary. Let $M$ be a $\Z G$-module whose restriction to every finite subgroup of $G$ is projective. Then
$$\widehat H^j(G,M\otimes_{\Z H}\Z G)=0$$ for all $j$ and all $\lhf$-subgroups $H$ of $G$.
\end{lemma}
\begin{proof}[Proof of Lemma \ref{new}.]
If $H$ is an $\hf$-group then this can be proved by induction on the ordinal height of $H$ in the $\hf$-hierarchy. The proof proceeds in exactly the same way as the proof of the Vanishing Theorem (\cite{cokr2}, \S8). 

In general, suppose that $H$ is an $\lhf$-group. 
Let $(H_\lambda)$ be the family of finitely generated subgroups of $H$.
Then we may view $H$ as the filtered colimit $\displaystyle H=\lim_\to H_\lambda$.
\end{proof}

Now suppose that $G$ is as in the statement of Theorem \ref{main}. Lemma \ref{new} shows that
$$\widehat H^0(G, B)=0.$$
and using the ring structure on $B$ it follows that
$$\widehat\Ext^0_{\Z G}(B,B)=0.$$
This implies that $B$ has finite projective dimension: it is a general fact that a module $M$ has finite projective dimension if and only if $\widehat\Ext^0_{\Z G}(M,M)=0$. Like the coinduced module, the module $B$ contains a copy of the trivial module $\Z$ in the form of the constant functions. Thus Theorem \ref{main} (i) is established.

Let $M$ be a module satisfying the hypotheses of (ii). By (i) we know that $M\otimes B$ has finite projective dimension and the proof that $M$ is projective requires two steps. First we show that $M$ is projective over $\Z H$ for all $\hf$-subgroups $H$ of $G$. The argument here is essentially the same as that used to prove Theorem B of \cite{cokr2}, using transfinite induction on the least ordinal $\alpha$ such that $H$ belongs to $\clh_\alpha\FF$. In the inductive step one considers an action of $H$ on a contractible finite dimensional complex $X$ so that all isotropy groups belong to $\clh_\beta\FF$ with $\beta<\alpha$. If $C_*\epi\Z$ denotes the augmented (reduced) cellular chain complex of $X$ then the inductive hypothesis shows that $M\otimes C_*\epi M$ is a projective resolution of $M$ of finite length and hence $M$ has finite projective dimension. Let $\overline B$ denote that quotient $B/\Z$ of $B$ by the constant functions. Then $M\otimes\underbrace{\overline B\otimes\dots\otimes\overline B}_k$ also has finite projective resolution for any $k\ge0$. Since $M$ arises as a $k$th kernel in a projective resolution of 
$M\otimes\underbrace{\overline B\otimes\dots\otimes\overline B}_k$
it follows that $M$ itself is projective of $\Z H$.

The $\hf$-subgroups of $G$ account for all countable subgroups. The next step is to establish by induction on the cardinality $\kappa$ that $M$ is projective on restriction to all subgroups of $G$ of cardinality $\kappa$. 
This argument can be found in the work \cite{benson1997} of Benson. In this way (ii) is established.

Part (iii) follows from 
the inequality
$$\silp(\Z G)\le\kappa(\Z G)$$ as stated in
Theorem C of \cite{cokr1}. Note that although (\cite{cokr1}, Theorem C) is stated for $\hf$-groups, the given proof shows that the above inequality holds for arbitrary groups.

Lemma \ref{mislin1994} yields (iv).

We are assuming that the complete cohomology is $0$-finitary in all dimensions. Now we also know that the ordinary cohomology coincides with the complete cohomology in high dimensions. Hence (v) is established. 

The trivial module $\Q$ is an instance of a module whose restriction to every finite subgroup has finite projective dimension, (projective dimension one in fact). Therefore the dimensional finiteness conditions imply (vi). This means in particular that 
$$\widehat H^0(G,\Q)=0.$$
Since the complete cohomology is $0$-finitary we can deduce that $\widehat H^0(G,\Z)$ is torsion. Being a ring with a one, it therefore has finite exponent, say m, and a simple argument with classical Tate cohomology shows that the orders of the finite subgroups of $G$ must divide $m$ and thus (vii) is established. The argument for proving (viii) can be found in \cite{krmi}. Although the Theorem as stated there does not directly apply to our situation, a reading of the proof will reveal that the all the essentials to make the construction work are already contained in the conclusions (i)--(vii).
\end{proof}

\begin{proof}[Proof of Theorem \ref{basic1}]
We first show that the complete cohomology of $G$ is $0$-finitary in all dimensions. 
Recall that the $j$th complete cohomology of $G$ is the colimit:
$$\widehat H^j(G,M):=\colimn H^{j+n}(G,\Omega^nM).$$
The maps $H^{j+n}(G,\Omega^nM)\to H^{j+n+1}(G,\Omega^{n+1}M)$ in this system are the connecting maps in the long exact sequence of cohomology which comes from the short exact sequence
$$\Omega^{n+1}M\mono F\Omega^n M\epi \Omega^nM.$$ Let $S=\{s\in\N:\ s+j\in\cont(G)\}$. Since $S$ is infinite, it is cofinal in $\N$. Hence
$$\widehat H^j(G,M):=\colims H^{j+s}(G,\Omega^sM).$$ Now, for any vanishing filtered colimit system $(M_\lambda)$ of $\Z G$-modules we have
\begin{eqnarray*}
\colim \widehat H^j(G,M_\lambda)&=&\colim\colims H^{j+s}(G,\Omega^sM_\lambda)\\
&=&\colims\colim H^{j+s}(G,\Omega^sM_\lambda)\\
&=&\colims H^{j+s}(G,\colim\Omega^sM_\lambda)\\
&=&0.
\end{eqnarray*}
Theorem \ref{basic1} now follows from Theorem \ref{main}.
\end{proof}

\section{General behaviour of finitary cohomology functors}

In this section we show how the finitary properties of one cohomology functor can influence neighbouring functors. Our arguments are based on an unpublished observation of Robert Snider. The first gives a further insight into the nature of the finite-cofinite dichotomy for the set $\mathcal F_0(G)$. It is a property held by many groups $G$ including all $\hof$-groups that $H^n(G,\blah)$ vanishes on free modules for all sufficiently large $n$. When this is so, there is a very simple proof that the finitary set is either finite or cofinite: it is a corollary of the following.

\begin{lemma} 
Let $n$ be a positive integer.
Suppose that $G$ is a group such that
\begin{enumerate}
\item $H^{n-1}(G,\blah)$ vanishes on all projective $\Z G$-modules, and
\item $H^n(G,\blah)$ is $0$-finitary.
\end{enumerate}
Then $H^{n-1}(G,\blah)$ is $0$-finitary.
\end{lemma}
\begin{proof}
Let $F$ and $\Omega$ denote the free module and loop functors described in the proof of Theorem \ref{basic1}. Let $(M_\lambda)$ be a vanishing filtered colimit system of $\Z G$-modules. From the short exact sequence
$$\Omega M_\lambda\to FM_\lambda\to M_\lambda$$
we obtain the long exact sequence
$$\dots\to H^{n-1}(G,FM_\lambda)\to H^{n-1}(G,M_\lambda)\to H^n(G,\Omega M_\lambda)\to\dots$$
Here the left hand group vanishes by hypothesis (i) and the right hand system vanishes on passage to colimit by hypothesis (ii). Hence
$$\colim H^{n-1}(G,M_\lambda)=0$$ as required.
\end{proof}

Thus, if $G$ is a group for which the set 
$$\{n:\ H^n(G,F)\textrm{ is non-zero for some free module }F\}$$ is bounded while the finitary set $\cont(G)$ is unbounded, then the finitary set is cofinite.

We have seen that any finite or cofinite set can be realized as the $0$-finitary set of some group. It is interesting to note that the existence of certain normal or near normal subgroups will impose some restrictions. The next lemma provides a way of seeing this.

\begin{lemma}\label{flat extension1}
Let $G$ be a group and suppose that there is an overring $R\supset \Z G$ such that $R$ is flat over $\Z G$ and $\Z\otimes_{\Z G}R=0$. Let $n$ be a positive integer. If both $H^{n-1}(G,\blah)$ and $H^{n+1}(G,\blah)$ are $0$-finitary then $H^{n}(G,\blah)$ is also $0$-finitary.
\end{lemma}
\begin{proof}
Let $F$ and $\Omega$ denote the free module and loop functors described in the proof of Theorem \ref{basic1}. Let $(M_\lambda)$ be a vanishing filtered colimit system of $\Z G$-modules. Then we have a short exact of vanishing filtered colimit systems:
$$\Omega M_\lambda\to FM_\lambda\to M_\lambda.$$
Applying the long exact sequence of cohomology and taking colimits we obtain the exact sequence
$$\colim H^n(G,FM_\lambda)\to\colim H^n(G,M_\lambda)\to\colim H^{n+1}(G,\Omega M_\lambda).$$
Here we wish to prove that the central group is zero and we know that the right hand term is zero because $H^{n+1}(G,\blah)$ is $0$-finitary. Therefore it suffices to prove that the left hand group is zero.
Since the $FM_\lambda$ are free we have the short exact sequence
$$FM_\lambda\to FM_\lambda\otimes_{\Z G}R\to (FM_\lambda\otimes_{\Z G}R)/FM_\lambda$$ of vanishing filtered colimit systems and hence we obtain an exact sequence
{\small $$\colim H^{n-1}(G,(FM_\lambda\otimes_{\Z G}R)/FM_\lambda)
\to \colim H^n(G,FM_\lambda)\to \colim H^n(G,FM_\lambda\otimes_{\Z G}R).$$}
We need to prove that the central group here is zero and we know that the left hand group vanishes because $H^{n-1}(G,\blah)$ is $0$-finitary. Therefore it suffices to prove that the right hand group is zero. In fact it vanishes even before taking colimits: let $F$ be any free $\Z G$-module and let $P_*\epi\Z$ be a projective resolution of $\Z$ over $\Z G$. Then
$\Hom_{\Z G}(P_*,F\otimes_{\Z G}R)\iso\Hom_R(P_*\otimes_{\Z G}R,F\otimes_{\Z G}R)$ is split exact because $R$ is flat over $\Z G$ and $\Z\otimes_{\Z G}R=0$.
Thus $H^*(G,F\otimes_{\Z G}R)=0$.
\end{proof}

For example, if $G$ is a group with a non-trivial torsion-free abelian normal subgroup $A$ then the lemma can be applied by taking $R$ to be the localization $\Z G(\Z A\setminus\{0\})^{-1}$ and shows that for such groups there cannot be isolated members in the complement of the finitary set $\cont(G)$. 
The condition that $A$ is normal can be weakened and yet it can still be possible to draw similar conclusions. We conclude this paper with two further results showing how this can happen.

Two subgroups $H$ and $K$ of a group $G$ are said to be commensurable if and only if $H\cap K$ has finite index in both $H$ and $K$. We write $Comm_G(H)$ for the set $\{g\in G:\ H\textrm{ and }H^g\textrm{ are commensurable}\}$. This is a subgroup of $G$ containing the normalizer of $H$.

\begin{lemma}
Let $G$ be a group with a subgroup $H$ such that $Comm_G(H)=G$ and $\Z H$ is a prime Goldie ring. Then the set $\Lambda$ of non-zero divisors in $\Z H$ is a right Ore set in $\Z G$. Moreover, if $H$ is non-trivial, then the localization $R:=\Z G\Lambda^{-1}$ satisfies the hypotheses of Lemma \ref{flat extension1}.
\end{lemma}
\begin{proof} 
Before starting, recall that in a prime Goldie ring, the set of non-zero divisors is a right Ore set and the resulting Ore localization is a simple Artinian ring.
We first prove that $\Lambda$ is a right Ore set in $\Z G$.
If $H$ is normal in $G$ then this is an easy and well known consequence of $\Lambda$ being a right Ore set in $\Z H$. Now consider the general case.
Let $r$ be an element of $\Z G$ and let $\lambda$ be an element of $\Lambda$. We need to find $\mu\in\Lambda$ and $s\in\Z G$ such that $r\mu=\lambda s$. Choose any way 
$$r=g_1r_1+\dots+g_mr_m$$
of expressing $r$ as a finite sum in which each $r_i$ belongs to $\Z H$ and $g_i\in G$. Since all the subgroups $g_iHg_i^{-1}$ are commensurable with $H$ we can choose a normal subgroup $K$ of finite index in $H$ such that 
$$K\subseteq\bigcap_{i=1}^mg_iHg_i^{-1}.$$
The group algebra $\Z K$ inherits the property of being a prime Goldie ring and the set of non-zero divisors in $\Z K$ is $$\Lambda_0:=\Lambda\cap\Z K.$$ By our initial remarks on the case of a normal subgroup, $\Lambda_0$ is a right Ore set in $\Z H$. Moreover, $\Z H\Lambda_0^{-1}$ is finitely generated over the Artinian ring $\Z K\Lambda_0^{-1}$. It follows {\em a fortiori} that
$\Z H\Lambda_0^{-1}$ is Artinian as a ring and since every non-zero divisor in an Artinian ring is a unit, we conclude that
$$\Z H\Lambda_0^{-1}=\Z H\Lambda^{-1}.$$
Hence, there exists a $t$ in $\Z H$ such that $\nu:=\lambda t\in\Lambda_0$: to see this, simply choose an expression $t\nu^{-1}$ for $\lambda^{-1}\in\Z H\Lambda^{-1}$ in the spirit of the localization
$\Z H\Lambda_0^{-1}$. For each $i$, we have $g_i^{-1}Kg_i\subseteq H$ and hence 
$g_i^{-1}\nu g_i\in\Z H$. 
It is straightforward to check that each $g_i\nu g_i^{-1}$ is a non-zero divisor in $\Z H$.
Applying the Ore condition to the pair $r_i,g_i^{-1}\nu g_i$ we find 
$s_i\in\Z H$ and $\mu_i\in\Lambda$ such that $$r_i\mu_i=g_i^{-1}\nu g_is_i.$$ It is routine that a finite list of elements in an Ore localization can be placed over a common denominator and it is therefore possible to make these choices so that the $\mu_i$ are all equal: we do this and write $\mu$ for the common element.
Thus
$$r_i\mu=g_i^{-1}\nu g_is_i,$$
and 
$$r\mu=\sum_ig_ir_i\mu=\sum_ig_ig_i^{-1}\nu g_is_i=\nu\left(\sum_ig_is_i\right)
=\lambda t\left(\sum_ig_is_i\right).$$
This establishes the Ore condition as required with $s=t\left(\sum_ig_is_i\right)$.

Finally, assume $H$ is non-trivial and let $\mathfrak h$ denote the augmentation ideal in $\Z H$. Then $\mathfrak h$ is non-zero and $\mathfrak h.\Z H\Lambda^{-1}$ is a non-zero two-sided ideal in the simple Artinian ring $\Z H\Lambda^{-1}$. Hence $\mathfrak h.\Z H\Lambda^{-1}=
\Z H\Lambda^{-1}$ and $\Z\otimes_{\Z H}\Z H\Lambda^{-1}=0$. It follows that $\Z\times_{\Z G}R=0$ so $R$ does indeed satisfy the hypotheses of Lemma \ref{flat extension1}.
\end{proof}

The condition $Comm_G(H)=G$ has been studied by the author in cohomological contexts, see \cite{kropholler-BS} and \cite{kropholler2006}. The second paper \cite{kropholler2006} addresses a more general situation in which $H$ is replaced by a set $\mathcal S$ of subgroups which is closed under conjugation and finite intersections: it is then shown one can define a cohomological functor $H^*(G/\mathcal S,\blah)$ on $\Z G$-modules and that spectral sequence arguments can be used to carry out certain calculations. Here we show, for the reader familiar with \cite{kropholler2006} how these arguments may be used to investigate when the new functors $H^*(G/\mathcal S,\blah)$ are finitary. 

\begin{lemma}
Let $\mathcal S$ and $G$ be as above. 
\begin{enumerate}
\item
The functor $H^0(G/\mathcal S,{\phantom M})$ is $0$-finitary.
\item
 If $G$ is finitely generated then the functor 
$H^1(G/\mathcal S,{\phantom M})$ is $0$-finitary.
\item More generally if $n$ is an integer such that $G$ has type $\fp_n$ and all members of $\mathcal S$ have type $\fp_{n-1}$ then the functors $H^i(G/\mathcal S,{\phantom M})$ are $0$-finitary for all $i\le n$.
\end{enumerate}
\end{lemma}

\begin{proof} 
We prove part (iii) by induction on $n$. The case $n=0$ is easy: this is part (i) of the statement and the finitary property is inherited from ordinary cohomology. The case $n=1$ is part (ii) of the statement and there is no need to treat this separately. Fix $n\ge1$ and assume inductively that the result is established for numbers $<n$. In particular we may assume that $H^i(G/\mathcal S,{\phantom M})$ is $0$-finitary when $i<n$.

Let $(M_\lambda)$ be a vanishing filtered colimit system in the category $\Mod\Z G/\mathcal S$. Taking colimits of the spectral sequences of \cite{kropholler2006} we obtain the spectral sequence
$$E_2^{p,q}=\colim H^p(G/\mathcal S,H^q(\mathcal S,M_\lambda))\implies \colim H^{p+q}(G,M_\lambda).$$ 
Now consider the cases when $p+q\le n$, $p\ge0$, $q\ge0$. When $p$ is less than $n$, the inductive and originally stated finitary assumptions imply that $E_2^{p,q}=0$ and so we have a block of zeroes on the $E_2$-page of the spectral sequence in the range $0\le p\le n-1$ and $0\le q\le n$. Therefore only the term $E_2^{n,0}=E_\infty^{n,0}$ and the cohomology $\colim H^{n}(G,M_\lambda)$ is isomorphic to $$\colim H^n(G/\mathcal S,H^0(\mathcal S,M_\lambda)).$$ 
But $\colim H^{n}(G,M_\lambda)$ is zero by assumption and so
 $$\colim H^n(G/\mathcal S,H^0(\mathcal S,M_\lambda))=0.$$ 
 The $M_\lambda$ were chosen in the subcategory so this simplifies to
 $$\colim H^n(G/\mathcal S,M_\lambda)=0.$$ 
 This vanishing applies to any choice of system $(M_\lambda)$ and thus $H^n(G/\mathcal S,{\phantom M})$ is $0$-finitary as required.
\end{proof}

\bibliographystyle{amsplain}
%\bibliography{Continuity}
%\end{document}

\section{Hamilton's Results}

\subsection{When is group cohomology finitary?}

Hamilton \cite{hamilton2011} uses the results of this paper to characterize the locally (polycyclic-by-finite) groups cohomology almost everywhere finitary: these are shown to be precisely the locally (polycyclic-by-finite) groups with finite virtual cohomological dimension and in which the normalizer of every non-trivial finite subgroup is of type $\fpinfty$. In particular this class of groups is subgroup closed. Note that the class of locally (polycyclic-by-finite) groups includes the class of abelian-by-finite groups and already, within the class of abelian-by-finite groups there are many interesting examples. The abelian group $\Q^{+}\times C_{2}$ (a direct product of the additive group of rational numbers by the cyclic group of order $2$ has almost all its cohomology functors infinitary. By contrast, the non-abelian extension of $\Q^{+}$ by $C_{2}$ is almost everywhere finitary even though it is infinitely generated. Hamilton finds that in general, the locally polycyclic-by-finite groups which have almost all cohomology functors finitary form a subgroup closed class. 

In view of our Theorem 2.1, Hamilton naturally focusses on groups with finite virtual cohomological dimension. He shows, for example, that if $G$ is a group with finite vcd and then $G$ has cohomology almost everywhere finitary over the field $\F_{p}$ of $p$ elements if and only if $G$ has finitely many conjugacy classes of elementary abelian $p$-subgroups and the centralizer of each non-trivial elementary abelian $p$-subgroup is of type $\fpinfty$ over $\F_{p}$.

Clearly, a natural question is whether one can generalize Hamilton's results from locally(polycyclic-by-finite) groups to other classes of soluble groups. One of the main reasons why this appears hard is that there is no clear classification of which soluble groups have type $\fpinfty$ over a given finite field. There are satisfactory theories of soluble groups of type $\fpinfty$ over $\Q$ and over $\Z$ but these have yet to be generalized to the case of finite fields.
Hamilton's proofs make use of the deep results \cite{henn} of Henn and in particular this leads to an answer to a question raised by Leary and Nucinkis \cite{learnnucinkis}, namely he shows that
 if $G$ is a group of type VFP over $\F_{p}$, and $P$ is a $p$-subgroup of $G$, then the centralizer $C_{G}(P)$ of $P$ is also of type VFP over $\F_{p}$. 

The most relevant questions raised by this research are as follows:

\begin{question}
Let $G$ be a soluble group and let $p$ be a prime. What are the homological and cohomological dimensions of $G$ over $\F_{p}$? Is there a simple criterion for $G$ to have type $\fpinfty$ over $\F_{p}$.
\end{question}

One may expect soluble groups to behave similarly over $\F_{p}$ as they do over $\Q$ with the obvious elementary caveat that one has to take care of $p$-torsion. However, there is no detailed account in the literature: Bieri's notes confine analysis to the characteristic zero case subsequent authors have studied this case alone in depth.

Finally, in this paper, Hamilton proves a more general result for groups that admit a finite dimensional classifying space for proper actions. He concludes that if $G$ is such a group and if there are just finitely many conjugacy classes of non-trivial finite subgroups for each of which the corresponding centralizers have cohomology almost everywhere finitary, then $G$ itself has cohomology almost everywhere finitary.

Hamilton uses results \cite{leary} of Leary to show that the converse of this result fails. Leary has constructed groups of type $\fpinfty$ which are of type VFP but which have infinitely many conjugacy classes of finite subgroups.

\subsection{Eilenberg--Mac Lane Spaces} In a second paper \cite{hamilton2009}, Hamilton studies the question of whether the property \emph{almost everywhere finitary} impacts on the Eilenberg--Mac Lane space of a group. Hamilton's main result (\cite{hamilton2009}, Theorem A) includes the statement that a group $G$ in the class $\lhf$ has cohomology almost everywhere finitary if and only if $G\times\Z$ (the direct product of $G$ with an infinite cyclic group) admits an Eilenberg--Mac Lane space with finitely many $n$-cells for all sufficiently large $n$. There are two natural questions arising from this research.

\begin{question}
Does Hamilton's (\cite{hamilton2009}, Theorem A) hold for arbitrary groups, outwith the class $\lhf$?
\end{question}

\begin{question}
Can Hamilton's (\cite{hamilton2009}, Theorem A) be proved with out the stabilization device of replacing $G$ by $G\times\Z$.
\end{question}

It is natural so speculate that both of these questions have a positive answer but they remain open.

\subsection{Group actions on spheres} In a third paper \cite{hamilton2008}, Hamilton builds on \cite{hamilton2011} by showing that in a locally(polycyclic-by-finite) group with cohomology almost everywhere finitary, every finite subgroup admits a free action on some sphere. This perhaps surprising fact is proved purely algebraically by showing that the same algebraic restrictions apply to the finite subgroups in Hamilton's context as apply in the theory of group actions on spheres, namely that subgroups of order a product of two (not necessarily distinct) primes must be cyclic. So the natural questions that arise are:

\begin{question}
Is there a geometric explanation for the connection between Hamilton's (\cite{hamilton2008}, Theorem 1.3) which explains the link with group actions on spheres? Are there similar results for a larger class of groups, for example, soluble groups of finite rank.
\end{question}

\bibliography{Continuity}

\providecommand{\bysame}{\leavevmode\hbox to3em{\hrulefill}\thinspace}
\providecommand{\MR}{\relax\ifhmode\unskip\space\fi MR }
% \MRhref is called by the amsart/book/proc definition of \MR.
\providecommand{\MRhref}[2]{%
  \href{http://www.ams.org/mathscinet-getitem?mr=#1}{#2}
}
\providecommand{\href}[2]{#2}
\begin{thebibliography}{10}

\bibitem{adamekrosicky}
Ji{\v{r}}{\'{\i}} Ad{\'a}mek and Ji{\v{r}}{\'{\i}} Rosick{\'y}, \emph{Locally
  presentable and accessible categories}, London Mathematical Society Lecture
  Note Series, vol. 189, Cambridge University Press, Cambridge, 1994.
  \MR{1294136 (95j:18001)}

\bibitem{benson1997}
D.~J. Benson, \emph{Complexity and varieties for infinite groups. {I}, {II}},
  J. Algebra \textbf{193} (1997), no.~1, 260--287, 288--317. \MR{1456576
  (99a:20054)}

\bibitem{bestvinabrady1997}
Mladen Bestvina and Noel Brady, \emph{Morse theory and finiteness properties of
  groups}, Invent. Math. \textbf{129} (1997), no.~3, 445--470. \MR{1465330
  (98i:20039)}

\bibitem{bieri-qmw}
Robert Bieri, \emph{Homological dimension of discrete groups}, second ed.,
  Queen Mary College Mathematical Notes, Queen Mary College Department of Pure
  Mathematics, London, 1981. \MR{715779 (84h:20047)}

\bibitem{Brown1975}
Kenneth~S. Brown, \emph{Homological criteria for finiteness}, Comment. Math.
  Helv. \textbf{50} (1975), 129--135. \MR{0376820 (51 \#12995)}

\bibitem{Brown-book}
\bysame, \emph{Cohomology of groups}, Graduate Texts in Mathematics, vol.~87,
  Springer-Verlag, New York, 1994, Corrected reprint of the 1982 original.
  \MR{1324339 (96a:20072)}

\bibitem{Brown1987}
Kenneth~S. Brown and Ross Geoghegan, \emph{An infinite-dimensional torsion-free
  {${\rm FP}_{\infty }$} group}, Invent. Math. \textbf{77} (1984), no.~2,
  367--381. \MR{752825 (85m:20073)}

\bibitem{cokr2}
Jonathan Cornick and Peter~H. Kropholler, \emph{Homological finiteness
  conditions for modules over strongly group-graded rings}, Math. Proc.
  Cambridge Philos. Soc. \textbf{120} (1996), no.~1, 43--54. \MR{1373347
  (96m:16009)}

\bibitem{cokr1}
\bysame, \emph{Homological finiteness conditions for modules over group
  algebras}, J. London Math. Soc. (2) \textbf{58} (1998), no.~1, 49--62.
  \MR{1666074 (99k:20105)}

\bibitem{hamilton2008}
Martin Hamilton, \emph{Finitary group cohomology and group actions on spheres},
  Proc. Edinb. Math. Soc. (2) \textbf{51} (2008), no.~3, 651--655. \MR{2465929
  (2009m:20084)}

\bibitem{hamilton2009}
\bysame, \emph{Finitary group cohomology and {E}ilenberg-{M}ac{L}ane spaces},
  Bull. Lond. Math. Soc. \textbf{41} (2009), no.~5, 782--794. \MR{2557459
  (2010k:55011)}

\bibitem{hamilton2011}
\bysame, \emph{When is group cohomology finitary?}, J. Algebra \textbf{330}
  (2011), 1--21. \MR{2774614}

\bibitem{henn}
Hans-Werner Henn, \emph{Unstable modules over the {S}teenrod algebra and
  cohomology of groups}, Group representations: cohomology, group actions and
  topology ({S}eattle, {WA}, 1996), Proc. Sympos. Pure Math., vol.~63, Amer.
  Math. Soc., Providence, RI, 1998, pp.~277--300. \MR{1603179 (99a:20056)}

\bibitem{houghton}
C.~H. Houghton, \emph{The first cohomology of a group with permutation module
  coefficients}, Arch. Math. (Basel) \textbf{31} (1978/79), no.~3, 254--258.
  \MR{521478 (80c:20073)}

\bibitem{kropholler-BS}
P.~H. Kropholler, \emph{Baumslag-{S}olitar groups and some other groups of
  cohomological dimension two}, Comment. Math. Helv. \textbf{65} (1990), no.~4,
  547--558. \MR{1078097 (91j:20095)}

\bibitem{kropholler2006}
\bysame, \emph{A generalization of the {L}yndon-{H}ochschild-{S}erre spectral
  sequence with applications to group cohomology and decompositions of groups},
  J. Group Theory \textbf{9} (2006), no.~1, 1--25. \MR{2195835 (2006k:20107)}

\bibitem{kropholler1993}
Peter~H. Kropholler, \emph{On groups of type {$({\rm FP})_\infty$}}, J. Pure
  Appl. Algebra \textbf{90} (1993), no.~1, 55--67. \MR{1246274 (94j:20051b)}

\bibitem{krmi}
Peter~H. Kropholler and Guido Mislin, \emph{Groups acting on finite-dimensional
  spaces with finite stabilizers}, Comment. Math. Helv. \textbf{73} (1998),
  no.~1, 122--136. \MR{1610595 (99f:20086)}

\bibitem{leary}
Ian~J. Leary, \emph{On finite subgroups of groups of type {VF}}, Geom. Topol.
  \textbf{9} (2005), 1953--1976 (electronic). \MR{2175161 (2006j:20058)}

\bibitem{learnnucinkis}
Ian~J. Leary and Brita E.~A. Nucinkis, \emph{Some groups of type {$VF$}},
  Invent. Math. \textbf{151} (2003), no.~1, 135--165. \MR{1943744
  (2003k:20086)}

\bibitem{leinster}
Tom Leinster, \emph{Higher operads, higher categories}, London Mathematical
  Society Lecture Note Series, vol. 298, Cambridge University Press, Cambridge,
  2004. \MR{2094071 (2005h:18030)}

\bibitem{Lueck-survey}
Wolfgang L{\"u}ck, \emph{Survey on classifying spaces for families of
  subgroups}, Infinite groups: geometric, combinatorial and dynamical aspects,
  Progr. Math., vol. 248, Birkh\"auser, Basel, 2005, pp.~269--322. \MR{2195456
  (2006m:55036)}

\bibitem{brianmatthews}
Brian Matthews, \emph{Homological methods for graded $\mathbf k$-algebras},
  Phd, University of Glasgow, 2007.

\bibitem{mislin1994}
Guido Mislin, \emph{Tate cohomology for arbitrary groups via satellites},
  Topology Appl. \textbf{56} (1994), no.~3, 293--300. \MR{1269317 (95c:20072)}

\end{thebibliography}
\bibliographystyle{plain}

\end{document}